\documentclass[leqno,12pt]{amsart}
\usepackage{amssymb}
\usepackage{}
\usepackage{bbm}
\usepackage{graphicx}
\usepackage{amsfonts}
\usepackage{mathrsfs}
\usepackage{amssymb,amsmath,latexsym,amsfonts,amsbsy,amsthm,mathtools,graphicx,CJKutf8,CJKnumb,CJKulem,color}
\usepackage{mathrsfs,cite}
\usepackage{colortbl}
\usepackage{txfonts}
\usepackage{stmaryrd}
\usepackage{appendix}

\usepackage{verbatim}
\usepackage{float}

\usepackage{yhmath}

\allowdisplaybreaks

\usepackage{pdfsync}
\numberwithin{equation}{section}

\allowdisplaybreaks

\def\R{\mathbb R}

\usepackage[left=1 in, right=1 in,top=1 in, bottom=1 in]{geometry}

\def\XXint#1#2#3{{\setbox0=\hbox{$#1{#2#3}{\int}$ }
\vcenter{\hbox{$#2#3$ }}\kern-.6\wd0}}

\newcommand{\beq}{\begin{equation}}
\newcommand{\eeq}{\end{equation}}
\newcommand{\ben}{\begin{eqnarray}}
\newcommand{\een}{\end{eqnarray}}
\newcommand{\beno}{\begin{eqnarray*}}
\newcommand{\eeno}{\end{eqnarray*}}
\newtheorem{theorem}{Theorem}[section]

\newtheorem{lemma}{Lemma}[section]
\newtheorem{proposition}{Proposition}[section]

\numberwithin{equation}{section}
\begin{document}
\title[]{on the Minkowski problem with respect\\
to a mixed Euclidean-Gaussian density}

\author{Stephanie Mui}
\address{Department of Mathematics, Georgia Institute of Technology, 686 Cherry ST NW, Atlanta, GA 30332, USA}
\email[S. Mui]{smui3@gatech.edu}

\author{Lei Qin}
\address{School  of Mathematics, Hunan University,  Changsha, 410082, China}
\email[L. Qin]{qlhnumath@hnu.edu.cn}

\author{Sinan Wang}
\address{School  of Mathematics, Hunan University,  Changsha, 410082, China}
\email[S.N. Wang]{wangsinan@hnu.edu.cn}

\date{\today}

\begin{abstract}
In this paper, we pose a new class of Minkowski problems corresponding to the mixed Euclidean-Gaussian volume. Using the flow method, we prove the existence of smooth normalized solutions for the corresponding Monge-Amp\`{e}re-type equation in the symmetric case. Then, by approximation, we obtain the existence of origin-symmetric solutions to the mixed Euclidean-Gaussian Minkowski problem.
\end{abstract}

\subjclass[2020]{35J60, 52A38, 52A40}

\keywords{Mixed Euclidean-Gaussian Minkowski problem; Smooth solution; Monge-Amp\`ere equation.}

\maketitle
\section{Introduction}
We consider a basic model given by the Cartesian product of two Euclidean spaces
\begin{equation*}
\mathbb{R}^n=\mathbb{R}^s \times\mathbb{R}^k =\{z=(x,y): x\in\mathbb{R}^s , y\in\mathbb{R}^k \},~~~~~~~n=s+k\geq1,
\end{equation*}
equipped with the product-type density
\begin{equation}\label{E:82}
d\mu=\frac{e^{-\frac{|x|^2}{2}}}{(2\pi)^{\frac{s}{2}}}dxdy,~~~~~~(x,y)\in\mathbb{R}^n.
\end{equation}
If $K\subset\mathbb{R}^n$ has $C^1$-boundary, then the weighted volume of the mixed Euclidean-Gaussian type is given by
\begin{equation}\label{E:001}
V_{\text{mix}}(K)=\frac{1}{(2\pi)^{\frac{s}{2}}}\int_K e^{-\frac{|x|^2}{2}}dz=\frac{1}{(2\pi)^{\frac{s}{2}}}\int_K e^{-\frac{1}{2}\sum_{i=1}^s |(z\cdot \mathbf{e}_i)|^2}dz,
\end{equation}
where $\{\mathbf{e}_1,\ldots,\mathbf{e}_n\}$ is the standard orthonormal basis. The corresponding notion of perimeter is given by
\begin{equation}\label{E:83}
P_{\text{mix}}(K)=\frac{1}{(2\pi)^{\frac{s}{2}}}\int_{\partial K} e^{-\frac{|x|^2}{2}}d\mathcal{H}^{n-1}(z)=\frac{1}{(2\pi)^{\frac{s}{2}}}\int_{\partial K} e^{-\frac{1}{2}\sum_{i=1}^s |(z\cdot \mathbf{e}_i)|^2}d\mathcal{H}^{n-1}(z).
\end{equation}
The relevant isoperimetric problem was first posed by Fusco-Maggi-Pratelli in \cite{FMP11}
, where they studied the existence, symmetry, and regularity of the minimizers. 
Inspired by the results in \cite{FMP11}, we consider the Minkowski problem related to the mixed Euclidean-Gaussian volume.

It is well-known that \eqref{E:001} is actually the Euclidean volume $V(K)$ when $s=0$. Variational calculations performed on this volume yields
\begin{equation}\label{E:111}
\lim_{t\rightarrow0}\frac{V(K+tL)-V(K)}{t}=\int_{\mathbb{S}^{n-1}}h_{L}(u)dS_{K}(u),
\end{equation}
where $h_{L}: \mathbb{S}^{n-1}\rightarrow \mathbb{R}$ is the support function of $L$ and $S_{K}$ is the surface area measure of $K$. 

Minkowski \cite{M97} first posed the classical Minkowski problem, read as following: \textit{given a finite Borel measure $\mu$ on $\mathbb{S}^{n-1}$, what are the  necessary and sufficient conditions on $\mu$ so that $\mu$ is the surface area measure $S(K, \cdot)$ of a convex body $K$ in $\mathbb{R}^n$?} The influence of this problem is widespread, with deep connections to functional analysis, differential geometry, and nonlinear partial differential equations. In recent years, an increasing number of mathematicians have dedicated themselves to investigating these types of questions. In the case of the classical Minkowski problem, when the given measure is either discrete or has a continuous density, Minkowski \cite{M97,M03} proved the existence and uniqueness of solutions to this problem by utilizing the variational method. Several years later, based on the same technique, Aleksandrov \cite{A38,A39} and Fenchel-Jessen \cite{FJ38} independently solved the problem for arbitrary measures. In the smooth case, from the perspective of partial differential equations, the classical Minkowski problem is equivalent to solving a degenerate fully nonlinear Monge-Amp\`{e}re equation on $\mathbb{S}^{n-1}$,
\[
\det(h_{ij}(u)+h(u)\delta_{ij})=f(u).
\]

Over the past three decades, numerous Minkowski-type problems have been introduced, similar to the one mentioned above. Each problem arises from asking an analogous question to the classical one, with the surface area measure replaced by other geometric measures arising from the differentiation of geometric invariants. In the context of smoothness, this problem is equivalent to other specific fully nonlinear elliptic PDEs of varying natures. Prominent Minkowski-type problems include the $L_p$ Minkowski problem (see \cite{CW00,HLYZ05,L93}), the logarithmic Minkowski problem (see \cite{BLYZ13}), the dual Minkowski problem (see \cite{HLYZ16}), and so on. For more information on these Minkowski-type problems and the state of the art in this field, one can refer to the survey by Huang-Yang-Zhang \cite{HYZ25}.

In this paper, we consider a new class of Minkowski problems, which can be posed for the measure in \eqref{E:001} for $0<s<n$. In particular, the relevant geometric invariant is known as the mixed Euclidean-Gaussian volume $V_{\text{mix}}(K)$ stated as in \eqref{E:001}. 

When $s=n$ in \eqref{E:001}, the weighted volume becomes
\begin{equation}\label{E:004}
\gamma_{n}(K)=\frac{1}{(2\pi)^{\frac{n}{2}}}\int_K e^{-\frac{|x|^2}{2}}dx,
\end{equation}
which is known as Gaussian probability measure. Unlike Lebesgue measure, the Gaussian probability measure is neither translation-invariant nor homogeneous, which is the main challenge when addressing the relevant Minkowski-type problem. Moreover, the density decays exponentially fast as $|x|\rightarrow
\infty$. Similar to the classical case, ``differentiating" \eqref{E:004} on the set of convex bodies yields
\[
\lim_{t\rightarrow0}\frac{\gamma_{n}(K+tL)-\gamma_{n}(K)}{t}=\int_{\mathbb{S}^{n-1}}h_{L}(u)dS_{\gamma_{n},K}(u),
\]
and this uniquely defines the Gaussian surface area measure $S_{\gamma_n,K}$ of $K$. When $K\in \mathcal{K}_{o}^n$ has $C^2$ boundary with positive curvature, $S_{\gamma_n,K}$ is absolutely continuous with respect to the spherical Lebesgue measure and can be written as
\begin{equation*}
dS_{\gamma_{n}, K}(u)={(2\pi)^{-\frac{n}{2}}}e^{-\frac{1}{2}(h^{2}_{K}(u)+|\nabla h_{K}(u)|^2)}\det (\nabla^{2}h_{K}(u)+h_{K}(u)I)\quad \text{on}\mathbb{\ \ } \mathbb{S}^{n-1},
\end{equation*}
where $h_K:\mathbb{S}^{n-1}\rightarrow \mathbb{R}$ is the support function of $K$, $\nabla$ and $\nabla^2$ are gradient and Hessian operators with respect to an orthonormal frame on $\mathbb{S}^{n-1}$, and $I$ is the identity matrix. Therefore, the corresponding Gaussian Minkowski problem emerges in \cite{HXZ21}: \textit{given a finite Borel measure $\mu$ on $\mathbb{S}^{n-1}$, what are the sufficient and necessary conditions on $\mu$ such that $\mu$ is the Gaussian surface area measure $S_{V_{n}, K}$ of a convex body $K$?} When the given measure $\mu$ has a density $d\mu=fdu$, then the Gaussian Minkowski problem is equivalent to solving the following Monge-Amp\`{e}re type equation on $\mathbb{S}^{n-1}$,
\[
{(2\pi)^{-\frac{n}{2}}}e^{-\frac{1}{2}(h^{2}_{K}(u)+|\nabla h_{K}(u)|^2)}\det(h_{ij}(u)+h(u)\delta_{ij})=f(u).
\]
For further results on the existence, uniqueness, regularity, and generalizations of the Minkowski problem for the Gaussian measure, see for example  \cite{FLX23, L22, FHX23, CHLZ23, KL2023, HQ24, FJL26, FL26, H25, Hu25}.

In this paper, we establish the variational formula of the mixed Euclidean-Gaussian volume $V_{\text{mix}}(K)$ in the class of convex bodies.  In particular, we show that
\begin{equation}
\lim_{t\rightarrow0}\frac{V_{\text{mix}}(K+tL)-V_{\text{mix}}(K)}{t}=\int_{\mathbb{S}^{n-1}}h_{L}(u)dS_{V_{\text{mix}},K}(u),
\end{equation}
for any convex bodies $K$ and $L$ containing the origin in their interiors and where $dS_{V_{\text{mix}},K}$ is the mixed Euclidean-Gaussian surface area measure. The proof is presented in Theorem \ref{thm5}.  From this, we pose the corresponding Minkowski problem.

\textbf{The mixed Euclidean-Gaussian Minkowski problem.} \textit{Given a finite Borel measure $\mu$ on $\mathbb{S}^{n-1}$, what are the necessary and sufficient conditions on $\mu$ so that $\mu$ is the mixed Euclidean-Gaussian surface area measure $S_{V_{\mathrm{mix}}, K}$ of a convex body $K$ with $o\in int K$? If $K$ exists, to what extent is it unique?}

When $K$ is sufficiently smooth, its mixed Euclidean-Gaussian surface area measure is absolutely continuous with respect to the spherical Lebesgue measure:
\begin{equation*}
dS_{V_{\text{mix}}, K}(u)=\frac{1}{(2\pi)^{\frac{s}{2}}}e^{-\frac{1}{2}\sum_{i=1}^s |(h_{K}(u)u+\nabla h_{K}(u))\cdot \mathbf{e}_i)|^2}\det (\nabla^{2}h_{K}(u)+h_{K}(u)I)du.
\end{equation*}

When the given measure $\mu$ has a density $d\mu=fdu$, the mixed Euclidean-Gaussian Minkowski problem is equivalent to solving the following Monge-Amp\`{e}re type equation on $\mathbb{S}^{n-1}$,
\begin{equation}\label{E:11}
{(2\pi)^{-\frac{s}{2}}}e^{-\frac{1}{2}\sum_{i=1}^s |(h_K (u)u+\nabla_{\mathbb{S}^{n-1}}h_K (u))\cdot {\bf e}_i |^2} \det(\nabla^2 h_K (u) +h_K (u)\delta_{ij})=f(u).
\end{equation}

We use the flow method to show the existence of solutions to the PDE in Theorem 1.1, assuming sufficient smoothness of the given data. This result is stated as follows:   
\begin{theorem}\label{thm2}
Let $f\in C^{2,\alpha}(\mathbb{S}^{n-1})$ be a positive even function on the sphere $\mathbb{S}^{n-1}$, $0<\alpha<1$. Then there exists a  $C^{4,\alpha}$ origin-symmetric convex body $K\subset\mathbb{R}^n$  such that
\begin{equation}\label{E:1.1}
\frac{1}{(2\pi)^{\frac{s}{2}}}e^{-\frac{1}{2}\sum_{i=1}^s |(h_K (u)u+\nabla_{\mathbb{S}^{n-1}}h_K (u))\cdot {\bf e}_i |^2} \det(\nabla^2 h_K (u) +h_K (u)\delta_{ij})=\tau f(u),
\end{equation}
where
\[
\tau=\frac{\int_{\mathbb{S}^{n-1}}e^{-\frac{1}{2}\rho^2\sum_{i=1}^s|u\cdot {\bf e}_i|^2}\rho^n du}{(2\pi)^{\frac{s}{2}}\int_{\mathbb{S}^{n-1}}f^{-1}hdx}.
\]
\end{theorem}

The flow method transforms a static geometric or PDE problem directly into a dynamic evolutionary process. By studying the long-term behavior
(existence, convergence, asymptotics) of this evolutionary process (geometric flow), one can prove
the existence, uniqueness, and other properties of the original static solution. The method can be traced
back to 1974, when Firey \cite{F74} proposed a model: the rate of change of the shape of a convex body is determined
by its surface area function. This can be regarded as an early application of ``curvature
flow” in geometry, although Firey primarily focused on the model itself rather than solving the
Minkowski problem. Later in 1985, Tso \cite{T85} introduced a weak formulation of the Gaussian curvature
flow. Although his primary focus was on the elliptic Monge-Amp\`{e}re equation, his methods provided
important insights for subsequent research. Several years later, Chou-Wang \cite{CW00} developed the theory
of the inverse Gaussian curvature flow, proving that convex surfaces maintain their convexity and
expand into spheres under this flow. This work laid a solid foundation for using flow methods to
resolve Minkowski-type problems related to Gaussian curvature functions. In 2006, Chou-Wang \cite{CW06} pioneered the use of flow methods to address various Minkowski problems. Since then, the flow method has been widely applied by mathematical researchers, see \cite{DL25, L22, SW25, CT25, BG25, LXZ22, CL21, BIS21, LL20, HH25} and so on.

Consequently, we show the existence of normalized solutions to mixed Euclidean-Gaussian Minkowski problem.

\begin{theorem}\label{thm98}
 Suppose $\mu$ is a finite Borel measure not concentrated in any closed hemisphere. Then there exists a convex body $K\in\mathcal{K}_e^n$ and a constant $\tau $ such that
\begin{equation*}
S_{V_{\mathrm{mix}},K}=\tau\mu,
\end{equation*}
where the normalization constant is 
\begin{equation*}
\tau=\frac{\int_{\mathbb{S}^{n-1}}e^{-\frac{1}{2}\rho^2\sum_{i=1}^s|u\cdot {\bf e}_i|^2}\rho^n du}{(2\pi)^{\frac{s}{2}}\int_{\mathbb{S}^{n-1}}f^{-1}hdu}.
\end{equation*}
\end{theorem}
Our proof proceeds by first proving Theorem \ref{thm2} via the flow method, and then approximating by smooth positive densities to obtain Theorem \ref{thm98}. This results in an independent proof of the existence of solutions and additionally gives information about the regularity. This is a new approach for this class of measures, and one can find further generalizations of this measure and their corresponding Minkowski problems in \cite{KL2023}.

The subsequent sections of this article are organized as follows. In section \ref{sec2}, we present various essential preliminaries and the framework for the flow method. In section \ref{sec3}, we calculate the variational formula of the mixed Euclidean-Gaussian volume. We demonstrate the a-priori estimates of solutions for the flow equation in Section 4. It is important to clarify that repeated indices in the main text signify summation. To simplify computations, the summation symbol will not be used henceforth. The constants in different lines may differ but will just denoted as $C$ for simplification.

\section{Preliminaries}
\label{sec2}
In this section, we will provide notation, a brief introduction to the theory of convex bodies, and a framework to the theory of geometric flows. One can refer to the book by Schneider \cite{S93} for more details on convex bodies.

\subsection{Convex body}
Let $\mathbb{R}^n$ be the $n$-dimensional Euclidean space. The unit sphere in $\mathbb{R}^n$ is defined as $\mathbb{S}^{n-1}$. $C(\mathbb{S}^{n-1})$ represents the space of continuous functions on $\mathbb{S}^{n-1}$ and will always be equipped with the max-norm metric
\begin{equation*}
    \|f-g\|_{\infty}=\max_{x\in\mathbb{S}^{n-1}}|f(x)-g(x)|,
\end{equation*}
for $f,g\in C(\mathbb{S}^{n-1})$. The subspace of positive continuous functions from $C(\mathbb{S}^{n-1})$ will be denoted as $C^+(\mathbb{S}^{n-1})$.

A convex body in $\mathbb{R}^n$ is a compact convex set with nonempty interior. Denote the class of convex bodies that contain the origin in their interiors in $\mathbb{R}^n$ and the class of origin-symmetric convex bodies in $\mathbb{R}^n$ by $\mathcal{K}_0^n$ and $\mathcal{K}_e^n$ respectively.

Let $K$ be a compact convex subset in $\mathbb{R}^n$. Then its support function $h_K$ is defined by
\begin{equation}
    h_K(x)=\max\{x\cdot y: y\in K\},~~~~x\in\mathbb{R}^n.
\end{equation}
The support function $h=h_K: \mathbb{R}^n\rightarrow\mathbb{R}$ is a continuous function and is homogeneous of degree 1. Suppose $K$ contains the origin in its interior, then the radial function $\rho=\rho_K: \mathbb{R}^n\backslash\{0\}\rightarrow\mathbb{R}$ is defined by
\begin{equation}
    \rho_K(x)=\max\{\lambda: \lambda x\in K\},~~~~x\in \mathbb{R}^n\backslash\{0\}.
\end{equation}
The radial function $\rho_K$ is a continuous function homogeneous of degree $-1$. 


For each $f\in C^+(\mathbb{S}^{n-1})$, the Wulff shape $[f]$ generated by $f$ is the convex body defined by
\begin{equation*}
    [f]=\{x\in\mathbb{R}^n: x\cdot u\leq f(u), \text{for all}~~ u\in \mathbb{S}^{n-1}\}.
\end{equation*}
It is obvious that $h_{[f]}\leq f$ and $[h_K]=K$ for each $K\in\mathcal{K}_0^n$.

Let $K_i\in\mathcal{K}_0^n$ be a sequence of convex bodies in $\mathbb{R}^n$. We say that $K_i$ converges to a compact convex subset $K\subset\mathbb{R}^n$ with respect to the Hausdorff metric provided that when $i\rightarrow\infty$,
\begin{equation*}
    \|h_{K_i}(u)-h_K(u)\|_{\infty}=\max_{u\in\mathbb{S}^{n-1}}|h_{K_i}(u)-h_K(u)|\rightarrow 0.
\end{equation*}
If $K$ contains the origin in its interior, the above formula is equivalent to
\begin{equation*}
    \|\rho_{K_i}(u)-\rho_K(u)\|_{\infty}=\max_{u\in\mathbb{S}^{n-1}}|\rho_{K_i}(u)-\rho_K(u)|\rightarrow 0,
\end{equation*}
as $i\rightarrow\infty$.

We use $\nu_K$ to denote the Gauss map that takes $x\in \partial K$ to its unique outer unit normal, and due to the convexity of $K$, and the map $\nu_K$ is almost everywhere defined on $\partial K$. We use $\nu_K^{-1}$ to denote the inverse Gauss map. Since $K$ is not assumed to be strictly convex, the map $\nu_K^{-1}$ is a set-valued map, and for each set $\eta\subset\mathbb{S}^{n-1}$, we have
\begin{equation*}
    \nu_K^{-1}(\eta)=\{x\in\partial K:~\text{there exists}~\nu\in\eta~~ \text{such that}~~~\nu~~~ \text{is an outer unit normal at}~x\}.
\end{equation*}
Occasionally, we will shorten the notation for the Gauss and inverse Gauss maps of singletons as follows. For $u\in \mathbb{S}^{n-1}$, if $\nu_K$ is well-defined at $\rho_K(u)u\in\partial K$, then we write $\alpha_K(u)$ for $\nu_K(\rho_K(u)u)$. We will be using the following well-known relationships between the support and radial functions of convex bodies.

\begin{lemma}[\cite{CL21}]\label{RelationshipSR}
Let $K\in \mathcal{K}_0^n$. Let $h$ and $\rho$ be the support function and radial function of $K$ respectively. Let also $x_{\mathrm{max}}$ and $\xi_{\mathrm{min}}$ be two points such that $h_K(x_{\mathrm{max}})=\mathop{\max}_{\mathbb{S}^{n-1}} h$ and $\rho(\xi_{\mathrm{\min}})=\mathop{\min}_{\mathbb{S}^{n-1}} \rho$. Then
\begin{gather*}
\mathop{\max}_{\mathbb{S}^{n-1}} h_K=\mathop{\max}_{\mathbb{S}^{n-1}} \rho_K\quad {and}\quad \mathop{\min}_{\mathbb{S}^{n-1}} h_K=\mathop{\min}_{\mathbb{S}^{n-1}} \rho_K,\\
h_K(x)\geq x\cdot x_{\max} h_{K} (x_{\max}),\quad \forall x\in \mathbb{S}^{n-1},\\
\rho_K(\xi)\xi \cdot \xi_{\min}\leq \rho_{K} (\xi_{\min}),\quad \forall x\in \mathbb{S}^{n-1}.
\end{gather*}

\end{lemma}

\subsection{Geometric flow}
Let $M$ be a smooth, closed, uniformly convex hypersurface in $\mathbb{R}^n$ enclosing the origin in its interior. $M$ is parametrized by the inverse spherical image $X=\nu_M^{-1}: \mathbb{S}^{n-1}\rightarrow M$. The support function $h$ of $M$ can be computed by
\begin{equation}\label{E:98}
    h(x)=\langle x,X(x)\rangle,
\end{equation}
where $x\in\mathbb{S}^{n-1}$ is the unit outer normal of $M$ at $X(x)$. Additionally, it is straightforward to verify that
\begin{equation*}
X(x)=h(x)x+\nabla h(x),
\end{equation*}
where $\nabla $ is the covariant derivative with respect to the standard metric $e_{ij}$ of $\mathbb{S}^{n-1}$. The second fundamental form of $M$ is given by
\[
\Pi_{ij}=\nabla_{ij} h+he_{ij},
\]
where $\nabla_{ij}=\nabla^2$ denotes the second-order covariant derivative with respect to $e_{ij}$. Based on a smooth local orthonormal frame on $\mathbb{S}^{n-1}$, the principal radii of curvature of $M$ are the eigenvalues of the matrix
\[
b_{ij}=\nabla_{ij} h +h\delta_{ij}.
\]
Let $u,x\in \mathbb{S}^{n-1}$ satisfy
\[
\rho(u)u=X(x)=h(x)x+\nabla h(x).
\]
It is well-known that
\[
\rho^2=h^2+|\nabla h|^2,\quad h=\frac{\rho^2}{\sqrt{\rho^2+|\nabla \rho|^2}}.
\]

Assume that $0<f\in C^{\infty}(\mathbb{S}^{n-1})$. Let $X_0:\mathbb{S}^{n-1}\rightarrow \mathbb{R}^n$ be a parametrization of a smooth, closed, and uniformly convex hypersurface $\mathcal{M}_0$ with the origin in its interior. In this paper, we will investigate the long-time behavior of the family of convex hypersurfaces $\mathcal{M}_t$ parameterized by the smooth map $X(\cdot,t):\mathbb{S}^{n-1}\times[0,T)\rightarrow\mathbb{R}^n$ that satisfies the initial boundary equation
\begin{equation}\label{E:97}
\begin{cases}
        \partial_t X(x,t)=f(\nu)e^{-\frac{1}{2}\sum_{i=1}^s|(X\cdot \bf{e}_i)|^2}\det(h_{ij}+h\delta_{ij})\langle X,\nu\rangle \nu-(2\pi)^{\frac{s}{2}}X(x,t)\tau(t)\\
        X(x,0)=X_0(x).
\end{cases}
\end{equation}
Here, $\nu$ is the unit outer normal vector of the hypersurface $\mathcal{M}_t$ at the point $X(\cdot,t)$, $h(x,t)$ is the support function of $\mathcal{M}_t$, $T$ is the maximal time for which the solution exists, and
\begin{equation}\label{Tau-function}
    \tau(t)=\frac{\int_{\mathbb{S}^{n-1}}e^{-\frac{1}{2}\sum_{i=1}^s|X\cdot {\bf e}_i|^2}|X|^n dx}{(2\pi)^{\frac{s}{2}}\int_{\mathbb{S}^{n-1}}f^{-1} \langle X,x\rangle dx}.
\end{equation}
It is worth noticing that the flow \eqref{E:97} can be reduced to a scalar PDE of the support function $h$ of $\mathcal{M}_t$ as follows
\begin{equation}\label{E:0002}
\begin{cases}
        \partial_t h(x,t)=f(x)h(x,t)e^{-\frac{1}{2}\sum_{i=1}^s|X\cdot \bf{e}_i|^2}\det(h_{ij}+h\delta_{ij})-(2\pi)^{\frac{s}{2}}h(x,t)\tau(t),\\
        h(x,0)=h_0(x),
\end{cases}
\end{equation}
where $h_0$ is the support function of $\mathcal{M}_0$.

Note that for $u\in \mathbb{S}^{n-1}$, $\rho(u,t)$ is the radial function of the convex body enclosed by $\mathcal{M}_t$ with
\begin{equation}\label{2.2-1}
    \rho(u,t)=(h(x,t)^2+|\nabla h(x,t)|^2)^{\frac{1}{2}}.
\end{equation}
Moreover, if we substitute the identity
\begin{equation*}
    \frac{1}{\rho(u,t)}\frac{\partial\rho(u,t)}{\partial t}=\frac{1}{h(x,t)}\frac{\partial h(x,t)}{\partial t}
\end{equation*}
into \eqref{E:0002}, it follows that the evolution equation of the radial function $\rho(u,t)$ is given by
\begin{equation}\label{E:0005}
    \begin{cases}
\partial_t \rho(u,t)=f(x)\rho(u,t) e^{-\frac{1}{2}\rho(u,t)^2\sum_{i=1}^s|u\cdot {\bf e}_i|^2}\det(h_{ij}+h\delta_{ij})-(2\pi)^{\frac{s}{2}}\rho(u,t)\tau(t),\\
        \rho(u,0)=\rho_0,
    \end{cases}
\end{equation}
where $x=x(u,t)$ is the unit outer normal vector of $M_t$ at the point $\rho(u,t)u$ and $\rho_0$ is the radial function of $\mathcal{M}_0$.

Let $\mathcal{M}_t$ be the convex hypersurface with the support function $h: \mathbb{S}^{n-1}\times[0,T)\rightarrow\mathbb{R}$. We define
\begin{equation}\label{E:0003}
    G(t)=\frac{1}{|\mu|}\int_{\mathbb{S}^{n-1}}h(x,t)d\mu(x),
\end{equation}
where $d\mu=f^{-1}(x)dx$ and $|\mu|=\mu(\mathbb{S}^{n-1})$ for a Borel measure $\mu$ on $\mathbb{S}^{n-1}$. In the following, we prove that along the flow \eqref{E:0002}, $G(t)$ in \eqref{E:0003} stays constant.

\begin{lemma}\label{lem 6}
    Along the flow \eqref{E:0002}, we have $G(t)=G(0)$, $\forall t\in [0,T)$.
\end{lemma}
\begin{proof}
Note that
\begin{equation}\label{E:0004}
h\det(h_{ij}+h\delta_{ij})dx=\rho(u)^ndu.
\end{equation}
Combining the flow \eqref{E:0002}, the definition \eqref{Tau-function} of $\tau(t)$, and  \eqref{E:0004}, we have
  \begin{align*}
      \frac{d}{dt}G(t)=&\frac{1}{|\mu|}\int_{\mathbb{S}^{n-1}}f(x)^{-1}\partial_{t}h(x,t)dx\\
      =&\frac{1}{|\mu|}\int_{\mathbb{S}^{n-1}}\left(he^{-\frac{1}{2}\sum_{i=1}^s|X\cdot {\bf e}_i|^2}\det(h_{ij}+h\delta_{ij})-(2\pi)^{\frac{s}{2}}h\tau(t)f(x)^{-1}\right)dx\\
      =&\frac{1}{|\mu|}\left[\int_{\mathbb{S}^{n-1}}he^{-\frac{1}{2}\sum_{i=1}^s|X\cdot {\bf e}_i|^2}\det(h_{ij}+h\delta_{ij})dx-(2\pi)^{\frac{s}{2}}\int_{\mathbb{S}^{n-1}}h\cdot\frac{\int_{\mathbb{S}^{n-1}}e^{-\frac{1}{2}\sum_{i=1}^s|X\cdot {\bf e}_i|^2}|X|^n dx}{(2\pi)^{\frac{s}{2}} f(x)\int_{\mathbb{S}^{n-1}}f^{-1}\langle X,x\rangle dx}dx\right]\\
      =&\frac{1}{|\mu|}\left[\int_{\mathbb{S}^{n-1}}e^{-\frac{1}{2}|\rho(u)|^2\sum_{i=1}^s|u\cdot {\bf e}_i|^2}\rho(u)^n du-\int_{\mathbb{S}^{n-1}}e^{-\frac{1}{2}\rho(u)^2\sum_{i=1}^s|u\cdot {\bf e}_i|^2}\rho(u)^n du\right]\\
      =&0.
  \end{align*}
\end{proof}

Let $K_t$ be the convex body enclosed by $\mathcal{M}_t$, then we show that the mixed Euclidean-Gaussian volume of this set is non-decreasing along the flow \eqref{E:0002}. 

\begin{lemma}\label{lem10-1}
    The function $V_{\mathrm{mix}}(K_t)$ is non-decreasing along the flow \eqref{E:0002}, that is
    \begin{equation*}
        \frac{d}{dt}V_{\mathrm{mix}}(K_t)\geq0,
    \end{equation*}
    and the equality holds if and only if $K_t$ satisfies
    \begin{equation*}
       f e^{-\frac{1}{2}\rho^2\sum_{i=1}^s|(u\cdot {\bf e}_i)|^2} \det(h_{ij}+h\delta_{ij})=(2\pi)^{\frac{s}{2}}\tau(t).
    \end{equation*}
\end{lemma}

\begin{proof}
Converting the definition of the mixed Euclidean-Gaussian volume to polar coordinates, we have
\begin{equation}\label{E:0001}
    V_{\text{mix}}(t) :=  V_{\text{mix}}(K_t)=(2\pi)^{-\frac{s}{2}}\int_{\mathbb{S}^{n-1}}\int_0^{\rho_{[h_t]}(u)}e^{-\frac{1}{2}r^2\sum_{i=1}^s|(u\cdot {\bf e}_i)|^2} r^{n-1}drdu.
\end{equation}
By \eqref{E:0005}, we deduce that
\begin{align*}
    (2\pi)^{\frac{s}{2}}\frac{d}{dt}V_{\text{mix}}(K_t)=&\int_{\mathbb{S}^{n-1}}e^{-\frac{1}{2}\rho^2\sum_{i=1}^s|u\cdot {\bf e}_i|^2}\rho^{n-1}\partial_t \rho du\\
    =&\int_{\mathbb{S}^{n-1}}e^{-\frac{1}{2}\rho^2\sum_{i=1}^s|u\cdot {\bf e}_i|^2}\rho^{n-1}\left[f\rho e^{-\frac{1}{2}\rho^2\sum_{i=1}^s|u\cdot {\bf e}_i|^2} \det(h_{ij}+h\delta_{ij})-(2\pi)^{\frac{s}{2}}\rho \tau(t)\right]du\\
    =&\int_{\mathbb{S}^{n-1}}fe^{-\frac{1}{2}\rho^2\sum_{i=1}^s|u\cdot {\bf e}_i|^2}\rho^n \det(h_{ij}+h\delta_{ij})du-(2\pi)^{\frac{s}{2}}\int_{\mathbb{S}^{n-1}}e^{-\frac{1}{2}\rho^2\sum_{i=1}^s|u\cdot {\bf e}_i|^2}\rho^n \tau(t)du\\
    =&\int_{\mathbb{S}^{n-1}}fe^{-\frac{1}{2}\rho^2\sum_{i=1}^s|u\cdot {\bf e}_i|^2}\rho^n \det(h_{ij}+h\delta_{ij})du\\
    &-\int_{\mathbb{S}^{n-1}}e^{-\frac{1}{2}\rho^2\sum_{i=1}^s|u\cdot {\bf e}_i|^2}\rho^n \frac{\int_{\mathbb{S}^{n-1}}e^{-\frac{1}{2}\rho^2\sum_{i=1}^s|u\cdot {\bf e}_i|^2}\rho^n du}{\int_{\mathbb{S}^{n-1}}f^{-1}hdx}du.
\end{align*}
By \eqref{E:0004} and the H\"{o}lder inequality, we get
\begin{align*}
    &(2\pi)^{\frac{s}{2}}\left(\int_{\mathbb{S}^{n-1}}f^{-1}hdx\right)\frac{d}{dt}V_{\text{mix}}(K_t)\\
    =&\int_{\mathbb{S}^{n-1}}f^{-1}hdx \int_{\mathbb{S}^{n-1}}{f}e^{-\frac{1}{2}\rho^2\sum_{i=1}^s|u\cdot {\bf e}_i|^2}\rho^n \det(h_{ij}+h\delta_{ij})du-\left[\int_{\mathbb{S}^{n-1}}e^{-\frac{1}{2}\rho^2\sum_{i=1}^s|u\cdot {\bf e}_i|^2}\rho^n du\right]^2 \\
    =&\int_{\mathbb{S}^{n-1}}\frac{f^{-1}\rho^n}{\det(h_{ij}+h\delta_{ij})}du \int_{\mathbb{S}^{n-1}}{f}e^{-\frac{1}{2}\rho^2\sum_{i=1}^s|u\cdot {\bf e}_i|^2}\rho^n \det(h_{ij}+h\delta_{ij})d\nu-\left[\int_{\mathbb{S}^{n-1}}e^{-\frac{1}{2}\rho^2\sum_{i=1}^s|u\cdot {\bf e}_i|^2}\rho^n du\right]^2\\
    =&\int_{\mathbb{S}^{n-1}}\left(\sqrt{\frac{f^{-1}\rho^n}{\det(h_{ij}+h\delta_{ij})}}\right)^2du \int_{\mathbb{S}^{n-1}}\left(\sqrt{{f}e^{-\frac{1}{2}\rho^2\sum_{i=1}^s|u\cdot {\bf e}_i|^2}\rho^n \det(h_{ij}+h\delta_{ij})}\right)^2du\\
    &-\left[\int_{\mathbb{S}^{n-1}}e^{-\frac{1}{2}\rho^2\sum_{i=1}^s|u\cdot {\bf e}_i|^2}\rho^n du\right]^2\\
    \geq&\left[\int_{\mathbb{S}^{n-1}}e^{-\frac{1}{2}\rho^2\sum_{i=1}^s|u\cdot {\bf e}_i|^2}\rho^n du\right]^2-\left[\int_{\mathbb{S}^{n-1}}e^{-\frac{1}{2}\rho^2\sum_{i=1}^s|u\cdot {\bf e}_i|^2}\rho^n du\right]^2\\
    =& 0.
\end{align*}
\end{proof}

\begin{proposition}\label{prop1}
Let $K,~L\in\mathcal{K}^n_0$, if $K\subset L$, then $V_{\mathrm{mix}}(K)\leq V_{\mathrm{mix}}(L)$.
\end{proposition}
\begin{proof}
The condition $K\subset L$ implies that $0<\rho_K(u)\leq \rho_L(u)$ for all $u\in \mathbb{S}^{n-1}$. Hence we have
\begin{align*}
&V_{\text{mix}}(L)-V_{\text{mix}}(K)\\
=&\frac{1}{(2\pi)^{\frac{s}{2}}}\int_{\mathbb{S}^{n-1}}\int_0 ^{\rho_L(u)} e^{-\frac{1}{2}r^2 \sum_{i=1}^s |u\cdot {\bf e}_i|^2}r^{n-1}drdu-\frac{1}{(2\pi)^{\frac{s}{2}}}\int_{\mathbb{S}^{n-1}}\int_0 ^{\rho_K(u)} e^{-\frac{1}{2}r^2 \sum_{i=1}^s |u\cdot {\bf e}_i|^2}r^{n-1}drdu\\
=&\frac{1}{(2\pi)^{\frac{s}{2}}}\int_{\mathbb{S}^{n-1}}\int_{\rho_K(u)} ^{\rho_L(u)} e^{-\frac{1}{2}r^2 \sum_{i=1}^s |u\cdot {\bf e}_i|^2}r^{n-1}drdu\\
\geq&0.
\end{align*}
\end{proof}

\section{the mixed euclidean-gaussian measure }\label{sec3}
In this section, we introduce the mixed Euclidean-Gaussian surface area measure via a variational formula and the corresponding mixed Euclidean-Gaussian Minkowski problem, as well as establish several fundamental properties.

Before presenting the definition, we will first recall a relevant Lemma from \cite{HLYZ16}.
\begin{lemma}\label{lem3}
Let $K\in\mathcal{K}_0^n$ and $f\in C(\mathbb{S}^{n-1})$. Suppose $\delta>0$ is small enough so that for each $t\in(-\delta,\delta)$, we have
\begin{equation*}
h_t=h_K+tf>0.
\end{equation*}
Then
\begin{equation*}
\lim_{t\rightarrow0}\frac{\rho_{[h_t]}(u)-\rho_K(u)}{t}=\frac{f(\alpha_K(u))}{h_K(\alpha_K(u))}\rho_K(u)
\end{equation*}
for almost every $u\in\mathbb{S}^{n-1}$ with respect to the spherical Lebesgue measure. Moreover, there exists an $M>0$, such that
\begin{equation*}
|\rho_{[h_t]}(u)-\rho_K(u)|<M|t|
\end{equation*}
for all $u\in\mathbb{S}^{n-1}$ and $t\in(-\delta,\delta)$.
\end{lemma}
With the help of Lemma \ref{lem3}, we deduce the variational formula of \eqref{E:001}.

\begin{theorem}\label{thm5}
Let $K\in\mathcal{K}_0^n$, and $f\in C(\mathbb{S}^{n-1})$. Then
\begin{equation*}
\lim_{t\rightarrow0}\frac{V_{\mathrm{mix}}([h_K +tf])-V_{\mathrm{mix}}(K)}{t}=\int_{\mathbb{S}^{n-1}}f(u)dS_{V_{\mathrm{mix}, K}}(u).
\end{equation*}
\end{theorem}
\begin{proof}
Denote $h_K +tf$ by $h_t$. Using polar coordinates, we have
\begin{equation*}
V_{\text{mix}}([h_t])=\frac{1}{(2\pi)^{\frac{h}{2}}}\int_{\mathbb{S}^{n-1}}\int_0 ^{\rho_{[h_t]}(u)} e^{-\frac{1}{2}r^2 \sum_{i=1}^s |u\cdot {\bf e}_i|^2}r^{n-1}drdu.
\end{equation*}
Fix $u\in \mathbb{S}^{n-1}$ and define
\begin{equation}\label{E:99}
F(\tau)=\int_0^{\tau} e^{-\frac{1}{2}r^2 \sum_{i=1}^s |u\cdot {\bf e}_i|^2}r^{n-1}dr,\quad \tau \in [0,\infty).
\end{equation}
By the mean value theorem and Lemma \ref{lem3}, we have
\begin{equation*}
|F(\rho_{[h_t]}(u))-F(\rho_K (u))|\leq|F^{'} (\theta)||\rho_{[h_t]}(u)-\rho_K (u)|<M|t||F^{'} (\theta)|,
\end{equation*}
where $\theta$ is between $\rho_K (u)$ and $\rho_{[h_t]}(u)$. Since $K\in\mathcal{K}_0^n$ and $f\in C(\mathbb{S}^{n-1})$, there exists $M_1 >0$ such that for $t$ close to 0, $[h_t]\subset M_1 B$ holds. By the definition of $F$, we have that $F'$ is bounded. Therefore, there exists an $M_2 >0$ such that for all $u\in \mathbb{S}^{n-1}$
\begin{equation*}
|F(\rho_{[h_t]}(u))-F(\rho_K (u))|<M_2 |t|.
\end{equation*}
Therefore, the dominated convergence theorem and Lemma \ref{lem3} yield
\begin{align*}
\lim_{t\rightarrow0}\frac{V_{\text{mix}}([h_K +tf])-V_{\text{mix}}(K)}{t}=&\lim_{t\rightarrow0}\frac{1}{(2\pi)^{\frac{s}{2}}}\frac{\int_{\mathbb{S}^{n-1}}F(\rho_{[h_t]}(u))du-\int_{\mathbb{S}^{n-1}}F(\rho_K (u))du}{t}\\
=&\frac{1}{(2\pi)^{\frac{s}{2}}}\int_{\mathbb{S}^{n-1}}\lim_{t\rightarrow0}\frac{F(\rho_{[h_t]}(u))-F(\rho_K (u))}{t}du\\
=&\frac{1}{(2\pi)^{\frac{s}{2}}}\int_{\mathbb{S}^{n-1}}f(\alpha_K (u))e^{-\frac{1}{2}\sum_{i=1}^s |\rho_K(u)(u\cdot {\bf e}_i)|^2}\frac{\rho_K^{n}(u)}{h_K (\alpha_K (u))}du\\
=&\frac{1}{(2\pi)^{\frac{s}{2}}}\int_{\partial K}f(\nu_K (x))e^{-\frac{1}{2}\sum_{i=1}^s |(x\cdot {\bf e}_i)|^2}dS_{K}(x)\\
=&\int_{\mathbb{S}^{n-1}}f(u)dS_{V_{\text{mix}},K}(u).
\end{align*}
\end{proof}

Below, we show that the mixed Euclidean-Gaussian surface area measure is weakly convergent with respect to the Hausdorff metric.
\begin{theorem}\label{thm3.3}
Let $K_j\in \mathcal{K}_0^n$ be a sequence of convex bodies such that $K_j$ converges to $K_0\in \mathcal{K}_0^n$ in the Hausdorff metric. Then $S_{V_{\mathrm{mix}}, K_j}$ converges to $S_{V_{\mathrm{mix}},K_0}$ weakly.
\end{theorem}
\begin{proof}
Since $K_i \in \mathcal{K}_o^n$ with $K_i\rightarrow K_0 \in \mathcal{K}_{o}^n$, there exists a constant $c_1>0$ such that $\frac{1}{c_1}B\subseteq K_{i}\subseteq c_1B$ for sufficiently large $i$. Thus, one has
\begin{equation}\nonumber
e^{-\frac{1}{2}\sum_{i=1}^s |\rho_{K_{i}}(u)u\cdot {\bf e}_i|^2}\rho_{K_{i}}^{n-1}(u)\rightarrow e^{-\frac{1}{2}\sum_{i=1}^s |\rho_{K_{0}}(u)u\cdot {\bf e}_i|^2}\rho_{K_{0}}^{n-1}(u), \ \ \mathrm {uniformly\  on\  } \mathbb{S}^{n-1}.
\end{equation}

Let $g\in C(\mathbb{S}^{n-1})$. Since $\nu_{K_i}\rightarrow \nu_{K_0}$ almost everywhere on $\mathbb{S}^{n-1}$ with respect to the spherical Lebesgue measure, then for almost every $u \in \mathbb{S}^{n-1}$, one has
\begin{equation}\nonumber
\frac{g(\alpha_{K_i}(u))}{u \cdot \alpha_{K_i}(u))} \rightarrow \frac{g(\alpha_{K_0}(u))}{u \cdot \alpha_{K_0}(u)}.
\end{equation}
By the definition of $S_{V_{\text{mix}}, K_j}$, we have
\begin{align*}
\int_{\mathbb{S}^{n-1}}g(u)dS_{V_{\text{mix}}, K_j}(u)&=\frac{1}{(2\pi)^{\frac{s}{2}}}\int_{\mathbb{S}^{n-1}}g(\alpha_{K_j} (u))e^{-\frac{1}{2}\sum_{i=1}^s |\rho_{K_j}(u)u\cdot {\bf e}_i|^2}\frac{\rho_{K_j}^{n}(u)}{h_{K_j} (\alpha_{K_j} (u))}du\\
&=\frac{1}{(2\pi)^{\frac{s}{2}}}\int_{\mathbb{S}^{n-1}}e^{-\frac{1}{2}\rho_{K_j}(u)^2\sum_{i=1}^s|u\cdot {\bf e}_i|^2}\frac{g(\alpha_{K_j}(u))}{\alpha_{K_j}(u)\cdot u}\rho_{K_j}^{n-1}(u)du\\
&\rightarrow \frac{1}{(2\pi)^{\frac{s}{2}}}\int_{\mathbb{S}^{n-1}}e^{-\frac{1}{2}\rho_{K_0}(u)^2\sum_{i=1}^s|u\cdot {\bf e}_i|^2}\frac{g(\alpha_{K_0}(u))}{\alpha_{K_0}(u)\cdot u}\rho_{K_0}^{n-1}(u)du\\
&=\int_{\mathbb{S}^{n-1}}g(u)dS_{V_{\text{mix}}, K_0}(u).
\end{align*}
\end{proof}
\begin{lemma}
Let $K\in\mathcal{K}^n_0$. Then the mixed Euclidean-Gaussian surface area measure $S_{V_{\mathrm{mix}},K}$ is absolutely continuous with respect to the surface area measure $S_K$. Moreover, when $K$ is sufficiently smooth, $S_{V_{\mathrm{mix}},K}$ is also absolutely continuous with respect to the spherical Lebesgue measure.
\begin{proof}
Let $\eta\subset\mathbb{S}^{n-1}$ be such that $S_K(\eta)=0$, or equivalently, $\mathcal{H}^{n-1}(\nu_K^{-1}(\eta))=0$, then we conclude that
\[
S_{V_{\text{mix}},K}=\frac{1}{(2\pi)^{\frac{s}{2}}}\int_{x\in\nu_K^{-1}(\eta)}e^{-\frac{1}{2}\sum_{i=1}^s |x\cdot {\bf e}_i|^2}d\mathcal{H}^{n-1}(x)=0.
\]
Similarly, let $\eta\subset\mathbb{S}^{n-1}$ be such that $\mathcal{H}^{n-1}(\eta)=0$, then
\begin{equation*}
S_{V_{\text{mix}},K}=\frac{1}{(2\pi)^{\frac{s}{2}}}\int_{\eta}e^{-\frac{1}{2}\sum_{i=1}^s |(h_{K}(u)u+\nabla h_{K}(u))\cdot {\bf e}_i|^2}\det(\nabla^2h_{K}+h_{K}I)du=0.
\end{equation*}
\end{proof}
\end{lemma}



\section{the mixed euclidean-gaussian minkowski problem }\label{sec4}

We obtain the following estimates to show the existence of solutions to the flow equation \eqref{E:0002}.

\subsection{$C^0$ and $C^1$ estimates }

\begin{lemma}\label{lem4.1}
Suppose $f \in C^{\infty}(\mathbb{S}^{n-1})$ and $0<h_0\in C^{\infty}(\mathbb{S}^{n})$ are uniformly convex even functions. Let $h(\cdot,t)$ be a positive, smooth, and uniformly convex solution to \eqref{E:0002}. Suppose $\mu$ is a nonzero finite even Borel measure not concentrated in any hemisphere with $d\mu =fdx$. Then there exists a constant $C>0$ independent of $t$ such that for every $t\in[0,T)$,
    \begin{equation*}
    1/C \leq h(\cdot,t)\leq C\quad on\ \  \mathbb{S}^{n-1},
    \end{equation*}
\[
1/C\leq \rho(\cdot,t)\leq C\quad on\ \  \mathbb{S}^{n-1}.
\]
\end{lemma}
\begin{proof}
By Lemma \ref{RelationshipSR}, it suffices to establish the upper and lower bounds of $h(\cdot,t)$. Without loss of generality, we choose $h_0$ to be the support function of $\mathbb{S}^{n-1}$. By the definition \eqref{E:0003} of $G(t)$ and Lemma \ref{lem 6}, we have that for each $t\in [0,T)$
\begin{equation}\label{E:0006}
    G(t)=\frac{1}{|\mu|}\int_{\mathbb{S}^{n-1}}h(u,t)d\mu(u)=G(0)=1>0.
\end{equation}
For each $t\in [0,T)$, write
\[
h(x_t,t)=\max_{u\in \mathbb{S}^{n-1}} h(u,t)=R_t
\]
for some $x_t\in \mathbb{S}^{n-1}$. Obviously, $R_t>0$ is the maximal scaling factor for which $R_t x_t\in \mathcal{M}_t=X(\mathbb{S}^{n-1},t)$. By the definition of the support function, we have that
\begin{equation}
    h(u,t)=\sup_{\xi\in \partial K_t} \langle u, \xi \rangle \geq R_t \langle x_t ,u \rangle_{+}=R_t \max\{\langle x_t ,u \rangle,0\}.
\end{equation}
Therefore, we have
\begin{equation}\label{E:00010}
    1=\frac{1}{|\mu|}\int_{\mathbb{S}^{n-1}}h(u,t)d\mu(u)\geq \frac{1}{|\mu|}R_t\int_{\mathbb{S}^{n-1}} \langle u,x_t \rangle_{+}d\mu(u)\geq \frac{1}{|\mu|}R_t\min_{x\in \mathbb{S}^{n-1}}\int_{\mathbb{S}^{n-1}}\langle u, x \rangle _{+}d\mu(u).
\end{equation}
Namely,
\[
h(\cdot,t)\leq R_t \leq \frac{|\mu|}{\mathop{\min}_{x\in \mathbb{S}^{n-1}} \int_{\mathbb{S}^{n-1}} \langle u,x   \rangle_{+} d\mu(u) }.
\]
So now it suffices to show that
\begin{equation}\label{E:0009}
    F(u)\coloneqq\frac{1}{|\mu|}\min_{x\in \mathbb{S}^{n-1}}\int_{\mathbb{S}^{n-1}}\langle u,x \rangle_{+}d\mu(u)  \geq C_0>0,
\end{equation}
in order to obtain that $h(\cdot, t)$ has a uniform upper bound independent of $t$.

The proof of \eqref{E:0009} is very similar to \cite[formula (3.19)]{FLX23}, which only relies on the fact that the support of $\mu$ is not concentrated on any closed hemisphere of $\mathbb{S}^{n-1}$. We include it below for convenience.

Using the fact that $\mu$ is not contained in any closed hemisphere with Proposition 1.3 in \cite{CW06}, we have
\begin{equation*}
    F(u)>0.
\end{equation*}
Secondly, since the support function is continuous on $\mathbb{S}^{n-1}$, if we take $u_i\in \mathbb{S}^{n-1}$ with $u_i\rightarrow u\in \mathbb{S}^{n-1}$ as $i\rightarrow +\infty$, then $\langle u_i,x \rangle_{+}\rightarrow \langle u,x \rangle_{+}$ uniformly on $\mathbb{S}^{n-1}$. That is,
\begin{equation*}
    |F(u_i)-F(u)|\leq\frac{1}{|\mu|}\int_{\mathbb{S}^{n-1}}|\langle u_i,x \rangle_{+}-\langle u,x \rangle_{+}|d\mu(u)\rightarrow0
\end{equation*}
as $i\rightarrow +\infty$. Therefore $F(u)$ is continuous on a compact set $\mathbb{S}^{n-1}$, and it follows that there exists a constant $C_0>0$ such that \eqref{E:0009} holds.

For the lower bound of $h(\cdot,t)$, we prove it by contradiction. Suppose the contrary that there exists a sequence of times $\{ t_i\in (0,T)\}$ such that
\[
h(u_i,t_i)=\mathop{\inf}_{u\in \mathbb{S}^{n-1}} h(u,t_i)\rightarrow 0\quad \mathrm{as}\ i\rightarrow +\infty, 
\]
where $u_i\in \mathbb{S}^{n-1}$. By the compactness of the sphere, there exists $\theta\in \mathbb{S}^{n-1}$, such that $u_i\rightarrow \theta$ (if necessary, we can choose a subsequence of $\{ u_{i_j}\}$). Let $K_i$ be the Wulff shape generated by $h(\cdot, t)$. Obviously, $K_i\in \mathcal{K}_{e}$. By the uniform upper bound of $h(\cdot,t)$, there exists an $R>0$ such that $K_i \subset B_R$ for every $i$. Therefore, we have
\[
K_i \subset B_R\cap \left\{x\in \R^n:\ |x\cdot u_i|\leq h(u_i,t_i)\right\}.
\]
By the continuity of $V_{\text{mix}}(\cdot)$, we have
\[
V_{\text{mix}}(K_i)\leq V_{\text{mix}} (B_R\cap \{ x\in \R^n:\ |x\cdot u_k|\leq h(u_k,t_k)\})\rightarrow 0,
\]
which is a contradiction since Lemma \ref{lem10-1} and the initial data imply that
\[
V_{\text{mix}}(K_t) \geq V_{\text{mix}}(B_R) >0.
\]
\end{proof}

\begin{lemma}\label{lem4.2}
  Under the same assumptions of Lemma \ref{lem4.1}, there exists a positive constant $C$ independent of $t$ such that for every $t\in [0,T)$,
    \begin{equation*}
        |\nabla h(\cdot,t)|\leq C\quad on\ \mathbb{S}^{n-1},
    \end{equation*}
     \begin{equation*}
        |\nabla \rho (\cdot,t)|\leq C \quad on\ \mathbb{S}^{n-1}.
    \end{equation*}
\end{lemma}

\begin{proof}
The first inequality follows from Lemma \ref{lem4.1} and smoothness. Let $u,x\in \mathbb{S}^{n-1}$ satisfy
\[
\rho(u,t)u=h(x,t)x+\nabla h(x,t).
\]
Note that
\[
\rho^2=h^2+|\nabla h|^2,\quad h=\frac{\rho^2}{\sqrt{\rho^2+|\nabla \rho|^2}}.
\]
And so once again by Lemma \ref{lem4.1}, we complete the proof.
\end{proof}

\begin{lemma}\label{lem4.3}
 Under the same assumptions of Lemma \ref{lem4.1}, there exists a positive constant $C$ independent of $t$ such that for every $t\in [0,T)$,
    \begin{equation*}
       1/C\leq\tau(t)\leq C .
    \end{equation*}
\end{lemma}

\begin{proof}
By the definition \eqref{Tau-function} of $\tau (t)$, Lemma \ref{lem4.1}, and Lemma \ref{lem4.2}, we immediately obtain the upper and lower bounds.
\end{proof}

\subsection{$C^2$ estimates}

We next derive an upper bound for the support function of the hypersurfaces $\mathcal{M}_t$ evolved by \eqref{E:0002}.

\begin{lemma}\label{lem10}
  Assume the same assumptions as in Lemma \ref{lem4.1}. Then there exists a positive constant $C$ independent of $t$ such that for every $t\in [0,T)$,
    \begin{equation*}
        \det(h(\cdot,t)_{ij}+h(\cdot,t)\delta_{ij})\leq C \quad \text{on}\ \mathbb{S}^{n-1}.
    \end{equation*}
\end{lemma}
\begin{proof}
 Let
    \begin{equation*}
        H:=fe^{-\frac{1}{2}\sum_{i=1}^s|(h(x,t)x+\nabla h(x,t))\cdot {\bf e}_i|^2}h\quad \mathrm{and}\quad  S:=\det(h_{ij}+h\delta_{ij}),\quad \forall x\in \mathbb{S}^{n-1}.
    \end{equation*}
Let also $M:=\frac{1}{2}\rho(u)^2$, where $u$ and $x$ are related by \eqref{2.2-1}. 

For $\varepsilon>0$ small enough, define the function
\begin{equation}\label{4.5}
    P:=\frac{1}{1-\varepsilon M}\frac{HS}{h}.
\end{equation}
For every fixed $t$, we know that $P(u,t)$ attains its maximum at some point in $\mathbb{S}^{n-1}$. Hence, after taking a suitable rotation, we can take an orthonormal frame such that the cofactor matrix $S^{ij}$  of $h_{ij}+h\delta_{ij}$ is diagonal at this maximal point. We then compute the following
\begin{equation}\label{E:00088}
\left\{
\begin{aligned}
&   0=P_i=\frac{1}{1-\varepsilon M}\left(\frac{HS}{h}\right)_i+\frac{\varepsilon M_i}{(1-\varepsilon M)^2}\left(\frac{HS}{h}\right),\\
& 0\geq P_{ij}
    =\frac{\left(\frac{HS}{h}\right)_{ij}}{1-\varepsilon M}+\frac{\varepsilon\left(\frac{HS}{h}\right)M_{ij}}{(1-\varepsilon M)^2},\\
&\partial_tP=\frac{\partial_t\left(\frac{HS}{h}\right)}{(1-\varepsilon M)}+\frac{\varepsilon\left(\frac{HS}{h}\right)\partial_tM}{(1-\varepsilon M)^2}.
\end{aligned}
\right.
\end{equation}
Here, $P_{ij}\leq 0$ means that it is a negative semi-definite matrix. Furthermore, we have
\begin{equation}\label{E:00066}
     \partial_tP-HS^{ij}P_{ij}
     =\frac{1}{(1-\varepsilon M)}\left[\partial_t\left(\frac{HS}{h}\right)-HS^{ij}\left(\frac{HS}{h}\right)_{ij}\right]+\frac{\varepsilon\left(\frac{HS}{h}\right)}{(1-\varepsilon M)^2}(\partial_tM-HS^{ij}M_{ij}).
\end{equation}
 To deduce the upper bound of \eqref{E:00066}. we estimate the terms $\partial_t\left(\frac{HS}{h}\right)-HS^{ij}\left(\frac{HS}{h}\right)_{ij}$ and $\partial_tM-HS^{ij}M_{ij}$.

Recalling \eqref{E:0002} and \eqref{2.2-1}, we have
\begin{equation}
  \left\{
\begin{aligned}
&   \partial_t M=HSh-2(2\pi)^{\frac{s}{2}}\tau M+\sum_{k}H_kh_kS+\sum_{k}HS_kh_k\nonumber,\\
&  M_i=hh_i+\sum_{k} h_kh_{ki},\\
& M_{ij}=h_jh_i+h(b_{ij}-h\delta_{ij})+\sum_{k}(b_{kj}-h\delta_{kj})(b_{ki}-h\delta_{ki})+\sum_{k} h_k(b_{ki}-h\delta_{ki})_j.
\end{aligned}
\right.
\end{equation}
By \cite[Appendix A]{F17}, we know that
\begin{equation}\label{E:00090}
  \left\{
\begin{aligned}
& S^{ij}(h_{ij}+h\delta_{ij})=(n-1)S,\\
& [\det b_{ij}]_k=S^{ij}b_{ijk}.
    \end{aligned}
\right.
\end{equation}
Then, we have 
\begin{align*}  HS^{ij}M_{ij}=
    HS^{ij}\sum_{k} b_{kj}b_{ki}-(n-1)HSh+\sum_{k}Hh_kS_k.\nonumber
\end{align*}
Consequently, 
\begin{align}\label{E:00098}
    \partial_tM-HS^{ij}M_{ij}
    =nHSh-2(2\pi)^{\frac{s}{2}}\tau M+\sum_{k}H_kh_kS-HS^{ij}\sum_{k}b_{kj}b_{ki}.
\end{align}
Next, we will proceed with computing $\partial_t\left(\frac{HS}{h}\right)-HS^{ij}\left(\frac{HS}{h}\right)$. Equations \eqref{E:0002} and \eqref{E:00090}  yield
\begin{equation}\nonumber
  \left\{
\begin{aligned}  & \left(\frac{HS}{h}\right)_i=\frac{(HS)_i}{h}-\frac{(HS)h_i}{h^2},\\
    &
     \left(\frac{HS}{h}\right)_{ij}=\frac{(HS)_{ij}}{h}-\frac{(HS)h_{ij}}{h^2}-2(\log h)_i\left(\frac{HS}{h}\right)_j,\\
     &  \partial_t\left(\frac{HS}{h}\right)=\frac{S\partial_tH}{h}+HS^{ij}\left(\frac{HS}{h}\right)_{ij}+(n-2)\left(\frac{HS}{h}\right)^2+2HS^{ij}(\log h)_i\left(\frac{HS}{h}\right)_j-(n-2)(2\pi)^{\frac{s}{2}}\tau\left(\frac{HS}{h}\right)
\end{aligned}
\right.
\end{equation}
Hence, we have
\begin{align}\label{4.10}
     \partial_t\left(\frac{HS}{h}\right)-HS^{ij}\left(\frac{HS}{h}\right)_{ij}
     \leq (n-1)\left(\frac{HS}{h}\right)^2-(n-1)(2\pi)^{\frac{s}{2}}\tau\left(\frac{HS}{h}\right)-\left(\frac{HS}{h}\right)\partial_tM\sum_{i=1}^{s}|u\cdot {\bf e}_i|^2,
\end{align}
where we use the fact that
\begin{align}
    2HS^{ij}(\log h)_i\left(\frac{HS}{h}\right)_j=
    -\frac{2\varepsilon(n-1) H^2S^2|\nabla h|^2}{h^2(1-\varepsilon M)} \leq 0,\nonumber
\end{align}
from \eqref{E:00088} and \eqref{E:00090}.

Now we recall the following in \cite{S93}, 
\begin{equation*}
   \left\{
\begin{aligned}  
&\sum_{i=1}^{n-1}\lambda_i\geq(n-1)\left(\prod_{i=1}^{n-1}\lambda_i\right)^{\frac{1}{n-1}},\\
&S=\det(h_{ij}+h\delta_{ij})=\prod_{i=1}^{n-1}\lambda_i,\\
&S^{ij}b_{kj}b_{ki}=S {\rm{tr}} b_{ij}=S\sum_{i=1}^{n-1}\lambda_i.
\end{aligned}  
\right.
\end{equation*}
So we have
\begin{equation*}   S^{ij}b_{kj}b_{ki}\geq(n-1)S^{1+\frac{1}{n-1}}.
\end{equation*}
Moreover,
by the definition of $H$, we have
\begin{equation}\label{4.12}
\frac{\varepsilon\left(\frac{HS}{h}\right)}{(1-\varepsilon M)^2}H_kh_kS=\frac{\varepsilon\left(\frac{HS}{h}\right)^2}{(1-\varepsilon M)^2}h_k^2+\frac{\varepsilon\left(\frac{HS}{h}\right)^2}{(1-\varepsilon M)^2}\frac{f_k}{f}h_kh+\frac{\varepsilon\left(\frac{HS}{h}\right)^2}{(1-\varepsilon M)^2}\rho\rho_kh_kh\leq CP^2.
\end{equation}
Therefore, by \eqref{E:00066},  \eqref{E:00098}, \eqref{4.10} and \eqref{4.12}, we have
\begin{align*}
     &\partial_tP-HS^{ij}P_{ij}\\
     \leq&\frac{1}{(1-\varepsilon M)}\left[(n-1)\left(\frac{HS}{h}\right)^2-(n-1)(2\pi)^{\frac{s}{2}}\tau\left(\frac{HS}{h}\right)-\left(\frac{HS}{h}\right)\partial_tM\sum_{i=1}^{s}|u\cdot {\bf e}_i|^2\right]\\
&+\frac{\varepsilon\left(\frac{HS}{h}\right)}{(1-\varepsilon M)^2}\left[nHSh-2(2\pi)^{\frac{s}{2}}\tau M+H_kh_kS-HS^{ij}b_{kj}b_{ki}\right]\\
     =&\frac{(n-1)\left(\frac{HS}{h}\right)^2}{(1-\varepsilon M)}-\frac{(n-1)(2\pi)^{\frac{s}{2}}\tau\left(\frac{HS}{h}\right)}{(1-\varepsilon M)}-\frac{\left(\frac{HS}{h}\right)^2}{(1-\varepsilon M)}\rho^2\sum_{i=1}^s|u\cdot {\bf e}_i|^2+\frac{\left(\frac{HS}{h}\right)}{(1-\varepsilon M)}(2\pi)^{\frac{s}{2}}\tau\rho^2\sum_{i=1}^s|u\cdot {\bf e}_i|^2\\
&+\frac{n\varepsilon\left(\frac{HS}{h}\right)^2}{(1-\varepsilon M)^2}h^2-\frac{(2\pi)^{\frac{s}{2}}\varepsilon\tau\left(\frac{HS}{h}\right)\rho^2}{(1-\varepsilon M)^2}+\frac{\varepsilon\left(\frac{HS}{h}\right)}{(1-\varepsilon M)^2}H_kh_kS-\frac{\varepsilon\left(\frac{HS}{h}\right)}{(1-\varepsilon M)^2}HS^{ij}b_{kj}b_{ki}\\
    =&-\left[\rho^2\sum_{i=1}^s|u\cdot {\bf e}_i|^2-n+1\right]\frac{\left(\frac{HS}{h}\right)^2}{(1-\varepsilon M)}+\frac{\left(\frac{HS}{h}\right)}{(1-\varepsilon M)}(2\pi)^{\frac{s}{2}}\tau\left[\rho^2\sum_{i=1}^s|u\cdot {\bf e}_i|^2-n+1\right]\\
    &+\frac{n\varepsilon\left(\frac{HS}{h}\right)^2}{(1-\varepsilon M)^2}h^2-\frac{(2\pi)^{\frac{s}{2}}\varepsilon\tau\left(\frac{HS}{h}\right)\rho^2}{(1-\varepsilon M)^2}+\frac{\varepsilon\left(\frac{HS}{h}\right)}{(1-\varepsilon M)^2}H_kh_kS-\frac{\varepsilon\left(\frac{HS}{h}\right)}{(1-\varepsilon M)^2}HS^{ij}b_{kj}b_{ki}\\
\leq&CP+C_1P^2+\frac{\varepsilon\left(\frac{HS}{h}\right)}{(1-\varepsilon M)^2}H_kh_kS-\frac{\varepsilon(n-1)\left(\frac{HS}{h}\right)}{(1-\varepsilon M)^2}HS^{1+\frac{1}{n-1}}\\
    \leq&CP+C_2P^2-\left[\frac{\left(\frac{HS}{h}\right)}{1-\varepsilon M}\right]^{2+\frac{1}{n-1}}H^{-\frac{1}{n-1}}h^{1+\frac{1}{n-1}}(1-\varepsilon M)^{\frac{1}{n-1}}\varepsilon(n-1)\\
    \leq&CP+C_2P^2-C_3P^{2+\frac{1}{n-1}},
\end{align*}
for some positive constants $C,C_1,C_2,C_3$. Since $S^{ij}$ is a positive definite diagonal matrix and $P_{ij}$ is negative semi-definite at the maximal point, we have $\sum_{i,j=1}^n S^{ij} P_{ij} \leq 0$. Consequently, by \eqref{E:00088}, we have
\begin{equation*}
    \partial_tP<0,
\end{equation*}
if we take $P$ sufficiently large. Therefore, we obtain that for every $t\in[0,T)$,
\begin{equation*}
    P(\cdot,t)<P(\cdot,0)\leq \tilde{C},\quad \mathrm{on}\ \mathbb{S}^{n-1}.
\end{equation*}
By \eqref{4.5} and Lemma \ref{lem4.1}, we have
\begin{equation*}
    \frac{1}{C}S\leq P\leq CS
\end{equation*}
for a positive constant $C$.
This implies that $S$ has a uniform upper bound.
\end{proof}

In the following, we prove that the eigenvalues of the matrix $b_{ij}$ are bounded by positive constants from both above and below.

Let $\mathcal{M}^{*}_t$ be the polar set of $\mathcal{M}_t=X(\mathbb{S}^{n-1},\cdot)$ and $h^{*}(\cdot,t)$ be the support function of $\mathcal{M}^{*}_t$. Since
\begin{equation}
    \rho(\cdot,t)=\frac{1}{h^* (\cdot,t)}
\end{equation}
and
\begin{equation}
    h(x,t)^{n+1}h^* (u,t)^{n+1}\det(h_{ij}+h\delta_{ij})\det(h^*_{ij}+h^*\delta_{ij})=1,
\end{equation}
we could reformulate the flow \eqref{E:0002} as
\begin{equation}\label{E:00020}
    \begin{cases}
        \partial_t h^{*}(u,t)=-fe^{-\frac{1}{2}\sum_{i=1}^s|(h^* (u,t)u+\nabla h^*)\cdot {\bf e}_i|^2}\frac{\rho^*(x,t)^{n+1}}{(h^*)^{n}\det(h^*_{ij}+h^*\delta_{ij})}+(2\pi)^{\frac{s}{2}}h^* (u,t)\tau(t),\\
        h(u,0)=\frac{1}{h_0},
    \end{cases}
\end{equation}
where
\begin{equation*}
    \rho^*(x,t)=\left(h^* (u,t)^2 +|\nabla h^* (u,t)|^2\right)^{\frac{1}{2}}
\end{equation*}
is the value of the radial function at
\begin{equation*}
    x=\frac{h^* u+\nabla h^*}{\left(h^* (u,t)^2 +|\nabla h^* (u,t)|^2\right)^{\frac{1}{2}}}\in \mathbb{S}^{n-1}.
\end{equation*}
Notice also that function $f$ takes its value at the point $x$. 

By Lemma \ref{lem10}, we deduce that there exists a positive constant $C$ independent of $t$ such that
\begin{equation}\label{New1}
\det(h^*_{ij}+h^*\delta_{ij})\geq C\quad \mathrm{on}\quad \mathbb{S}^{n-1}.
\end{equation}
Therefore, by the above estimate, it is enough to establish the upper bound.

In the proof of following lemma, we use the notation $h$ instead of $h^*$ and $b$ instead of $b^*$ for simplicity.

\begin{lemma}\label{lem2}
There is a positive constant $C$ independent of $t$ such that for $i=1,\cdots,n-1$, the principal radii  of curvature $\lambda_i$ of the hypersufaces $M_t$ satisfy
\begin{equation}\label{E:11}
\frac{1}{C}\leq\lambda_i(\cdot,t)\leq C\quad \text{on} ~\mathbb{S}^{n-1}\times[0,T).
\end{equation}
\end{lemma}

\begin{proof}
Consider the auxiliary function with constants $A$ and $B$ to be determined later,
$$
P(u,t)=\lambda_{\max }\left(b_{ij}\right)h^A \exp\left( B|\nabla h|^2 \right), \quad \forall (u,t) \in \mathbb{S}^{n-1}\times[0,T).
$$
Here, $-A,B$ are positive constants to be choosen later. For any fixed $T'\in (0,T)$, assume $P(u,t)$ attains its maximum at $ (u_0,t_0) \in \mathbb{S}^{n-1}\times[0,T']$. By choosing a suitable orthonormal frame, we may assume $\left\{b_{ij}(u_0,t_0) \right\}$ is diagonal and $\lambda_{\max }\left(b_{i j}\right)\left(u_0,t_0 \right)=b_{11}\left(u_0,t_0\right)$.

Let $e_1, e_2, \cdots, e_{n-1}$ be the aforementioned orthonormal frame at the point $(u_0,t_0)\in\mathbb{S}^{n-1}\times[0,T']$. Then one has at $(u_0,t_0)$,
\begin{equation}\label{E:00022}
\begin{cases}
0=(\log P)_i=b^{11} b_{11 i}+A \frac{h_i}{h}+2 B \sum_{k} h_k h_{k i},\\
0\leq\frac{\partial\log P}{\partial t}=b^{11}(\partial_t h_{11}+\partial_t h)+A\frac{\partial_t h}{h}+2B\sum_{k} h_k\partial_t h_k. 
\end{cases}
\end{equation}
Without loss of generality, we can assume $t_0>0$. Moreover, $b^{ij}(\log P)_{ij}\leq 0$. A direct calculation yields
\begin{align*}
2Bb^{i j} \sum_{k}h_{k j} h_{k i}=2Bb_{ii}-4B(n-1)h+2Bh^2 b^{ii}.
\end{align*}
According to the Ricci identity \cite{BA87}, we have
\[
b_{i j11}=b_{11i j}-\delta_{ij} b_{11}+\delta_{11} b_{ij}-\delta_{i 1} b_{1j}+\delta_{1 j} b_{1i} .
\]
Then we have 
\begin{align}\label{E:4}
b^{ij}(\log P)_{ij}
=&b^{11}b^{ij} b_{11 i j}-\left(b^{11}\right)^2 b^{ij} b_{11 i} b_{11 j}+\frac{A(n-1)}{h}-\frac{A}{h^2}b^{ij}h_i h_j\\
&-(A+2B|\nabla h|^2)b^{ii}+2Bb^{ij}\sum_{k} h_{k j} h_{k i}+2Bb^{ij} \sum_{k} h_k b_{i j k}\nonumber\\
=& b^{11}b^{ij} b_{11 i j}-\left(b^{11}\right)^2 b^{ij} b_{11 i} b_{11 j}+\frac{A(n-1)}{h}-\frac{A}{h^2}b^{ij}h_i h_j\nonumber\\
&-(A+2B|\nabla h|^2-2Bh^2)b^{ii}-4B(n-1)h+2Bb_{i i}+2Bb^{ij} \sum_{k} h_k b_{i j k}.\nonumber\\
=&b^{11}b^{ij} [b_{i j 11}+\delta_{ij} b_{11}-\delta_{11} b_{ij}+\delta_{i 1} b_{1j}-\delta_{1 j} b_{1i}]-\left(b^{11}\right)^2 b^{ij} b_{11 i} b_{11 j}\nonumber\\
&+\frac{A(n-1)}{h}-\frac{A}{h^2}b^{ij}h_i h_j-(A+2B|\nabla h|^2-2Bh^2)b^{ii}\nonumber\\
&-4B(n-1)h+2Bb_{ii}+2Bb^{ij} \sum_{k}h_k b_{i j k}\nonumber.
\end{align}
By Lemma \ref{lem4.1}, at the maximal point $(u_0,t_0)$, we can select an  $A<0$ such that
\begin{equation}\label{E:412}
\begin{aligned}
b^{ij}(\log P)_{ij}
\geq& b^{11}b^{ii} b_{ii11}-\left(b^{11}\right)^2 b^{ii} b_{11 i} b_{11i}+\frac{A(n-1)}{h}-\frac{A}{h^2}b^{ii}h^2_i \\
&-(A+2B|\nabla h|^2-2Bh^2)b^{ii}-4B(n-1)h+2Bb_{i i}+2Bb^{ii} \sum_{k}h_k b_{i i k}\\
\geq& b^{11}b^{ii}b_{i i11}-\left(b^{11}\right)^2 b^{ii} b_{11 i} b_{11 i}-(A+2B|\nabla h|^2-2Bh^2)b^{ii}\\
&-2B b_{ii}+ 2B  b^{ii} \sum_{k} h_k b_{iik}-CB.
\end{aligned}
\end{equation}
Let
\[
\Lambda(u,t)=\log \left(fe^{-\frac{1}{2}\sum_{i=1}^s|(h (u,t)u+\nabla h)\cdot {\bf e}_i|^2}\frac{\rho^{n+1}}{h^{n}}\right)=\log \left(fe^{-\frac{1}{2}\rho^2(x,t)\sum_{i=1}^s|x\cdot {\bf e}_i|^2}\frac{\rho^{n+1}}{h^{n}}\right),
\]
where $x,u\in \mathbb{S}^{n}$ satisfy
\[
    \rho(x,t)=\left(h(u,t)^2 +|\nabla h (u,t)|^2\right)^{\frac{1}{2}}.
\]
A direct calculation yields
\begin{equation}\label{0616-2}
\begin{aligned}
\Lambda_k&=\frac{f_k}{f}+(n+1)\frac{\rho_k}{\rho}-n\frac{h_k}{h} -\sum_{i=1}^s\left((h (u,t)u+\nabla h)\cdot {\bf e}_i\right) \left( h_k u  +h\nabla_k u+\nabla h_k\right)\cdot {\bf e}_i\\
&=\frac{f_k}{f}+(n+1)\frac{\rho_k}{\rho}-n\frac{h_k}{h} -\sum_{i=1}^s\left((h (u,t)u+\nabla h)\cdot {\bf e}_i\right) \left( h_k u  +he_k+\nabla h_k\right)\cdot {\bf e}_i\\
&=\frac{f_k}{f}+(n+1)\frac{h_k b_{kk}}{\rho}-n\frac{h_k}{h} -\sum_{i=1}^s\left((h (u,t)u+\nabla h)\cdot {\bf e}_i\right) \left( h_k u  +he_k+\nabla h_k\right)\cdot {\bf e}_i
\end{aligned}
\end{equation}
Moreover, 
\begin{equation}\label{0616-3}
\begin{aligned}
\Lambda_{11}=&\frac{f_{11}f-f^2_1}{f^2}+(n+1)\frac{\left(h_{11}b_{11}+h_1b_{111}\right)\rho-\rho_1h_1b_{11}}{\rho^2}-n\frac{h_{11}h-h_1^2}{h^2}-\sum_{i=1}^s\left| \left( h_{1} u  +he_1+\nabla h_1\right)\cdot {\bf e}_i \right|^2\\
&-\sum_{i=1}^s\left((h (u,t)u+\nabla h)\cdot {\bf e}_i\right) \left( h_{11} u +2h_1 e_1 +h_1 \nabla_{e_1}{e_1}+\nabla h_{11}\right)\cdot {\bf e}_i.\\
=&\frac{f_{11}f-f^2_1}{f^2}+(n+1)\frac{\left(h_{11}b_{11}+h_1b_{111}\right)\rho-h^2_1 b^2_{11}}{\rho^2}-n\frac{h_{11}h-h_1^2}{h^2}-\sum_{i=1}^s\left| \left( h_{1} u  +he_1+\nabla h_1\right)\cdot {\bf e}_i \right|^2\\
&-\sum_{i=1}^s\left((h (u,t)u+\nabla h)\cdot {\bf e}_i\right) \left( h_{11} u +2h_1  e_1+\nabla h_{11}\right)\cdot {\bf e}_i\\
&\geq C_1 (1+b_{11}+ b_{11}^2) + \sum_{k=1}^{s} C(k)b_{11k}.
\end{aligned}
\end{equation}
By \eqref{E:00020}, we have
\begin{equation}\label{E:00021}
    \log[(2\pi)^{\frac{s}{2}}h\tau(t)-\partial_t h]= \Lambda-\log S,
\end{equation}
where $S:=\det(h_{ij}+h\delta_{ij})$. Then differentiating \eqref{E:00021} at the point $(u_0,t_0)$ with respect to $e_k$, 
we have
\begin{equation}\label{ADD20200601-1}
\frac{(2\pi)^{\frac{s}{2}}\tau  h_k -\partial_t h_k}{(2\pi)^{\frac{s}{2}}h\tau-\partial_t h}=\Lambda_k-b^{ij}  b_{ijk}.
\end{equation}
And we differentiate \eqref{E:00021} with respect to $e_1$ twice,
\begin{equation}\label{ADD12-2}
\begin{aligned}
\frac{(2\pi)^{\frac{s}{2}}\tau  h_{11} -\partial_t h_{11}}{(2\pi)^{\frac{s}{2}}h\tau-\partial_t h}&=\left(\frac{(2\pi)^{\frac{s}{2}}\tau  h_1 -\partial_t h_1}{(2\pi)^{\frac{s}{2}}h\tau-\partial_t h}\right)^2+\Lambda_{11}-b^{ii} b_{ii11}+b^{ii}b^{jj}(b_{ij1})^2\\
&\geq \Lambda_{11}-b^{ii} b_{ii11}+b^{ii}b^{jj}(b_{ij1})^2.
\end{aligned}
\end{equation}
At the point $(u_0,t_0)$, \eqref{ADD20200601-1} yields
\begin{equation}\label{ADD20200601-1-1}
\frac{\sum_{k}h_k\partial_t h_k}{(2\pi)^{\frac{s}{2}}h\tau-\partial_t h}=\frac{(2\pi)^{\frac{s}{2}}\tau  \sum_k h^2_k }{(2\pi)^{\frac{s}{2}}h\tau-\partial_t h}-\sum_{k}h_k\Lambda_k+b^{ii}\sum_{k} h_k b_{iik}.
\end{equation}
Dividing \eqref{E:00022} by $(2\pi)^{\frac{s}{2}}h\tau-\partial_t h$, we have
\begin{equation}\label{ADD12-1}
\begin{aligned}
&\frac{b^{11}(\partial_t h_{11}+\partial_t h)}{(2\pi)^{\frac{s}{2}}h\tau-\partial_t h}+A\frac{\partial_t h}{h\left( (2\pi)^{\frac{s}{2}}h\tau-\partial_t h\right)}+2B\frac{\sum_{k} h_k\partial_t h_k}{(2\pi)^{\frac{s}{2}}h\tau-\partial_t h}\\
&=\frac{b^{11}\left(\partial_t h_{11}-(2\pi)^{\frac{s}{2}}\tau h_{11}+(2\pi)^{\frac{s}{2}}\tau b_{11}-(2\pi)^{\frac{s}{2}}\tau h+\partial_t h\right)}{(2\pi)^{\frac{s}{2}}h\tau-\partial_t h}\\
&\quad+A\frac{\partial_t h}{h\left( (2\pi)^{\frac{s}{2}}h\tau-\partial_t h\right)}+2B\frac{\sum_{k} h_k\partial_t h_k}{(2\pi)^{\frac{s}{2}}h\tau-\partial_t h}\\
&=\frac{b^{11}(\partial_t h_{11}-(2\pi)^{\frac{s}{2}}\tau h_{11})}{(2\pi)^{\frac{s}{2}}h\tau-\partial_t h}+2B\frac{\sum_{k} h_k\partial_t h_k}{(2\pi)^{\frac{s}{2}}h\tau-\partial_t h}\\
&\quad-\frac{A}{h}+\frac{\tau(A+1)(2\pi)^{\frac{s}{2}}}{ (2\pi)^{\frac{s}{2}}h\tau-\partial_t h}+b^{11}\geq 0.
\end{aligned}
\end{equation}
By \eqref{E:412}, \eqref{ADD12-2}, \eqref{ADD20200601-1}, \eqref{ADD12-1} and \eqref{ADD20200601-1-1} we have
\begin{equation}\label{0616-4}
\begin{aligned}
 b^{ij}(\log P)_{ij} \geq  &b^{11}\left(-\frac{(2\pi)^{\frac{s}{2}}\tau  h_{11} -\partial_t h_{11}}{(2\pi)^{\frac{s}{2}}h\tau-\partial_t}+\Lambda_{11}+b^{ii}b^{jj}(b_{ij1})^2\right)-b^{ii} \left( A\frac{h_i}{h}+2B\sum_k h_k h_{ki}\right)\\
 &+2B \left( \sum_{k} h_k \Lambda_k+\frac{\sum_{k}h_k\partial_t h_k}{(2\pi)^{\frac{s}{2}}h\tau-\partial_t h}-\frac{(2\pi)^{\frac{s}{2}}\tau  \sum_k h^2_k }{(2\pi)^{\frac{s}{2}}h\tau-\partial_t h}\right)\\
 &-(A+2B|\nabla h|^2-2Bh^2)b^{ii}+2B b_{ii}-CB\\
\geq & b^{11}\left(\Lambda_{11}+b^{ii}b^{jj}(b_{ij1})^2\right)-b^{ii} \left( A\frac{h_i}{h}+2B\sum_k h_k h_{ki}\right)\\
 &+2B  \sum_{k} h_k \Lambda_k-\frac{(2\pi)^{\frac{s}{2}}\tau  \sum_k h^2_k }{(2\pi)^{\frac{s}{2}}h\tau-\partial_t h}+\frac{A}{h}-\frac{\tau(A+1)(2\pi)^{\frac{s}{2}}}{ (2\pi)^{\frac{s}{2}}h\tau-\partial_t h}-b^{11}.
\end{aligned}
\end{equation}
Since
\begin{align*}
\text{tr}\ b_{ij} \sum_{j,k} b^{jj}b^2_{ijk}&\geq \left( \sum_k b_{kk}\right) \left( \sum_k b^{kk}b^2_{ikk} \right)\\
&\geq \left( \sum_{k} \sqrt{b_{kk}} \sqrt{b^{kk}b^2_{ikk}}\right)^2\\
&= \left( \sum_k |b_{kki}|\right)^2\\
&\geq \left( \sum_k b^2_{kki}\right)^2,
\end{align*}
it follows from \eqref{0616-4} that
\begin{equation}
\begin{aligned}
0\geq & b^{11}\left(\Lambda_{11}+b^{ii}b^{jj}(b_{ij1})^2\right)-b^{ii} \left( A\frac{h_i}{h}+2B\sum_k h_k h_{ki}\right)\\
 &+2B \left( \sum_{k} h_k \Lambda_k-\frac{(2\pi)^{\frac{s}{2}}\tau  \sum_k h^2_k }{(2\pi)^{\frac{s}{2}}h\tau-\partial_t h}\right)+\frac{A}{h}-\frac{\tau(A+1)(2\pi)^{\frac{s}{2}}}{ (2\pi)^{\frac{s}{2}}h\tau-\partial_t h}-b^{11}\\
\geq & b^{11}\left(\Lambda_{11}+b^{ii}b^{jj}(b_{ij1})^2\right)+2B  \sum_{k} h_k \Lambda_k-b^{ii} \left( A\frac{h_i}{h}+2B\sum_k h_k h_{ki}\right)\\
 &-\frac{(2\pi)^{\frac{s}{2}}\tau(A+1-2B \sum_k h^2_k) }{(2\pi)^{\frac{s}{2}}h\tau-\partial_t h}+\frac{A}{h}-b^{11}\\
 \geq & b^{11}\Lambda_{11}+2B  \sum_{k} h_k \Lambda_k-b^{ii} \left( A\frac{h_i}{h}+2B h_ib_{ii}+2Bh_i h \right)\\
 &-\frac{(2\pi)^{\frac{s}{2}}\tau(A+1-2B \sum_k h^2_k) }{(2\pi)^{\frac{s}{2}}h\tau-\partial_t h}+\frac{A}{h}-b^{11}.
\end{aligned}
\end{equation}
By Lemma \ref{lem4.2}, we have
\[
|\nabla h|\leq C,
\]
for a positive constant $C$ independent of $t$. So, taking $A=-n+2BC^2$, we conclude that
\[
A+1-2B\sum_{k} h^2_k <0.
\]
It follows from \eqref{0616-4} that
\begin{equation}\label{0616-5}
b^{11}\Lambda_{11}+2B  \sum_{k} h_k \Lambda_k-b^{ii} \left( A\frac{h_i}{h}+2B\sum_k h_k h_{ki}\right)+\frac{A}{h}-b^{11} \leq 0.
\end{equation}
Note that
\begin{equation*}
   \left\{
\begin{aligned}  
&b^{ii} h_{ii}=b^{ii}(b_{ii}-h)=n-1-h\ \text{tr} b^{-1},\\
&\sum_{k}b^{ii} h_{ki}^{2}= b^{ii}(b_{ii}^2-2hb_{ii}+h^2)=-2(n-1)h+\text{tr}b+h^2\text{tr}b^{-1},\\
&\sum_{k}b^{ii}h_k h_{kii}=\sum_{k} h_k b^{ii}b_{iik}-b^{ii}h_i^2.
\end{aligned}  
\right.
\end{equation*}
By \eqref{0616-2}, \eqref{0616-3} and \eqref{E:00022}, we have
\begin{align*}
&b^{11}\Lambda_{11}+2B\sum_{k} h_k \Lambda_k \\
&=b^{11}\left(\frac{f_{11}f-f^2_1}{f^2}+(n+1)\frac{\rho_{11}\rho-\rho_1^2}{\rho^2}-n\frac{h_{11}h-h_1^2}{h^2}-\sum_{i=1}^s\left| \left( h_{1} u  +he_i+\nabla h_1\right)\cdot {\bf e}_i \right|^2 \right)\\
&\quad -b^{11}\left(\sum_{i=1}^s\left((h (u,t)u+\nabla h)\cdot {\bf e}_i\right) \left( h_{11} u +2h_1  e_1+\nabla h_{11}\right)\cdot {\bf e}_i \right)\\
&\quad + 2B h_k\left( \frac{f_k}{f}+(n+1)\frac{h_k b_{kk}}{\rho}-n\frac{h_k}{h} -\sum_{i=1}^s\left((h (u,t)u+\nabla h)\cdot {\bf e}_i\right) \left( h_k u  +he_k+\nabla h_k\right)\cdot {\bf e}_i \right)\\ 
&\geq b^{11} \left( C_1 (1+b_{11}+ b_{11}^2) + \sum_{k} C(k) b_{11k} \right) + C_3B\left(1+\sum_{k} b_{kk}\right)\\
&\geq C_1 (b^{11}+1+ b_{11}) + \sum_{k} C(k) \left( -A\frac{h_k}{h}-2Bh_kb_{kk}+2Bh\right)  + C_3B\left(1+\sum_{k} b_{kk}\right)\\
&\geq C b^{11}+C b_{11} +C(-A+B+1)-CB \sum_{k} b_{kk}.
\end{align*}
Hence, by the above inequality and \eqref{0616-5}  we have
\begin{equation*}
    \left(-\frac{h_i}{h}A\right)b^{ii}+C^2Bb_{ii}-CB-C(1+b^{11}+Bb_{11})\leq 0,
\end{equation*}
If we take
\begin{equation*}
    B>\max\left\{\frac{C}{2C^2},\frac{1+2C}{2-3C}\right\}
\end{equation*}
in the above inequality, then $b_{11}(u_0,t_0)$ is bounded from above. Combining this with Lemma \ref{lem4.1} and Lemma \ref{lem4.2}, it is clear that the principal radii of curvatures $\lambda^{*}_i$  of $M_t^*$ have a uniform upper bound. Together with \eqref{New1}, the proof is complete.
\end{proof}

We finish this section with the higher order estimates and the existence for all times.

\begin{lemma}\label{lem4.6}
The solutions to the flow equation \eqref{E:0002} exist for all times $t\in(0,\infty)$ in the class $C(\mathbb{S}^{n-1}\times [0,t])$ with
\[
\| h \|_{C^{m+2}(\mathbb{S}^{n-1}\times [0,\infty)])} \leq C\quad \forall m \geq 1.
\]
\end{lemma}

\begin{proof}
 Since \eqref{E:0002} is a parabolic equation, one obtains the short-time existence of its solutions via the implicit function theorem. Then Lemma \ref{lem4.1}, Lemma \ref{lem4.2}, Lemma \ref{lem4.3}, and Lemma \ref{lem2} imply that the flow equation \eqref{E:0002} is uniformly parabolic on the interval $[0,T)$. Therefore, by standard parabolic theory and Krylov-Safonov estimates \cite{KS81}, the smooth solutions of the flow equation \eqref{E:0002} exist for all time.
\end{proof}

\subsection{Proof of Main Theorem}

First, we prove Theorem \ref{thm2} below.

\begin{proof}[\bf Proof of the Theorem \ref{thm2}]
By Lemma \ref{lem10-1}, Lemma \ref{lem4.1}, and \eqref{E:0001}, we know that $V_{\text{mix}}(t)$ is non-decreasing and is uniformly bounded above with respect to $t$. In particular, there exists a positive constant $C$ independent of $t$ such that for any $t\geq0$,
\begin{equation*}
    V_{\text{mix}}(t)\leq C.
\end{equation*}
Hence for any $t\geq0$,
\begin{equation*}
    \int_0^{t}\frac{d}{dt}V_{\text{mix}}(t)dt=V_{\text{mix}}(t)-V_{\text{mix}}(0)\leq V_{\text{mix}}(t)\leq C,
\end{equation*}
which means that 
\begin{equation*}
     \int_0^{\infty}\frac{d}{dt}V_{\text{mix}}(t)dt\leq C.
\end{equation*}
Combining this with $\frac{d}{dt}V_{\text{mix}}(t)\geq0$ implies there exists a subsequence $t_j\rightarrow\infty$ such that 
\begin{equation*}
    \frac{d}{dt}V_{\text{mix}}(t_j)\rightarrow0,
\end{equation*}
which implies $\frac{\partial}{\partial t}h(\cdot, t_j)\rightarrow0$ in the flow equation \eqref{E:0002} . Moreover, from Lemma \ref{lem4.6}, we deduce that $h(\cdot, t_j)$ will converge to a smooth function $h_{\infty}$, which solves equation \eqref{E:1.1}.
\end{proof}

\begin{proof}[\bf Proof of Theorem \ref{thm98}]
 Let $0<f_j\in C^{\infty}(\mathbb{S}^{n-1})$ be such that $d\mu_j=f_jdu$ converges weakly to $d\mu$ as $j\rightarrow \infty$, and $d\mu$ is not concentrated on a hemisphere. Let $h_j$ be a solution of the following PDE
\begin{equation}\label{E:87}
\frac{1}{(2\pi)^{\frac{s}{2}}}e^{-\frac{1}{2}\sum_{i=1}^s |(h_K (u)u+\nabla h_K (u))\cdot {\bf e}_i |^2} \det(\nabla^2 h_K (u) +h_K (u)\delta_{ij})=f(u)\tau,
\end{equation}
where
\begin{equation*}
\tau=\frac{\int_{\mathbb{S}^{n-1}}e^{-\frac{1}{2}\rho^2\sum_{i=1}^s|u\cdot {\bf e}_i|^2}\rho^n du}{(2\pi)^{\frac{s}{2}}\int_{\mathbb{S}^{n-1}}f^{-1}hdu}
\end{equation*}
for $f = f_j$, where each $f_j$ is positive and smooth. Then
\begin{equation*}
\int_{\mathbb{S}^{n-1}}dS_{V_{\text{mix}, K_j}}=\int_{\mathbb{S}^{n-1}}\tau_jd\mu_j,
\end{equation*}
where
\begin{equation*}
\tau_j=\frac{\int_{\mathbb{S}^{n-1}}e^{-\frac{1}{2}\rho_j^2\sum_{i=1}^{s}|(\nu\cdot {\bf e}_i)|^2}\rho_j^n d\nu}{(2\pi)^{\frac{s}{2}}\int_{\mathbb{S}^{n-1}}f^{-1}_jh_jdu}.
\end{equation*}
Our goal is to prove
\begin{equation}\label{E:91}
\int_{\mathbb{S}^{n-1}}dS_{V_{\text{mix}},K_j}\rightarrow\int_{\mathbb{S}^{n-1}}dS_{V_{\text{mix}},K}.
\end{equation}
According to Theorem \ref{thm3.3}, it is sufficient to demonstrate that $K_j\rightarrow K$ in the Hausdorff metric to establish \eqref{E:91}. Hence, the Blaschke Selection Theorem plays a crucial role, which means the next step is to obtain an $R$ independent of $j$ such that $K_j\subset B_R$.

We claim that
\begin{equation}\label{E:88}
M_j:=\frac{1}{|\mu_j|}\min_{u_j\in\mathbb{S}^{n-1}}\int_{\mathbb{S}^{n-1}}\langle u_j, \hat{u}\rangle_{+}d\mu_j(\hat{u})\geq C_0>0.
\end{equation}
Proceeding by contradiction, suppose that there exists a subsequence of ${M_j}$, denoted by $M_{j^{'}}$, converging to 0. In particular, for some $u_{j^{'}}\in\mathbb{S}^{n-1}$,
\begin{equation}\label{E:89}
M_{j^{'}}:=\frac{1}{|\mu_{j^{'}}|}\min_{u_{j^{'}}\in\mathbb{S}^{n-1}}\int_{\mathbb{S}^{n-1}}\langle u_{j^{'}}, \hat{u}\rangle_{+}d\mu_{j^{'}}(\hat{u})\rightarrow0.
\end{equation}
Since $\mathbb{S}^{n-1}$ is compact, there exists a subsequence $u_{j^{''}}$ of $u_{j^{'}}$ such that $u_{j^{''}}\rightarrow u$. Take
\begin{equation*}
M_{j^{''}}:=\frac{1}{|\mu_{j^{''}}|}\min_{u_{j^{''}}\in\mathbb{S}^{n-1}}\int_{\mathbb{S}^{n-1}}\langle u_{j^{''}}, \hat{u}\rangle_{+}d\mu_{j^{''}}(\hat{u}).
\end{equation*}
On the one hand, $M_{j^{''}}\rightarrow0$ by assumption \eqref{E:89}. On the other hand, since the support function is continuous on $\mathbb{S}^{n-1}$, the condition $u_{j^{''}}\rightarrow u$ implies $\langle u_{j^{''}},\hat{u} \rangle_{+}\rightarrow \langle u,\hat{u} \rangle_{+}$ uniformly on $\mathbb{S}^{n-1}$. Therefore,
\begin{equation*}
M_{j^{''}}\rightarrow\frac{1}{|\mu|}\min_{u\in\mathbb{S}^{n-1}}\int_{\mathbb{S}^{n-1}}\langle u, \hat{u}\rangle_{+}d\mu(\hat{u})=0
\end{equation*}
which contradicts
\begin{equation*}
\frac{1}{|\mu|}\min_{u\in\mathbb{S}^{n-1}}\int_{\mathbb{S}^{n-1}}\langle u, \hat{u}\rangle_{+}d\mu(\hat{u})\geq C_0>0
\end{equation*}
by \eqref{E:0009}. Finally, \eqref{E:00010} and \eqref{E:88}  yield that $h_{K_j}$ has a uniform upper bound, that is
\begin{equation*}
h_{K_j}\leq r\frac{|\mu_j|}{\min_{u_j\in\mathbb{S}^{n-1}}\int_{\mathbb{S}^{n-1}}\langle u_j, \hat{u}\rangle_{+}d\mu_j(\hat{u})}\leq\frac{r}{C_0}.
\end{equation*}

Next, we fix our attention to the lower bound of $h_{K_j}$. Proceed by contradiction and assume that there exists some $u_0\in\mathbb{S}^{n-1}$ such that $h_{K_j}(u_0)\rightarrow0$. Lemma \ref{lem10-1} implies
\begin{equation*}
V_{\text{mix}}(K_t)=V_{\text{mix}}(t)\geq V_{\text{mix}}(0)=\frac{1}{2}.
\end{equation*}
Hence, if $K_{\infty}$ is a solution of the equation \eqref{E:87} with $K_t\rightarrow K_{\infty}$, then Lemma \ref{lem4.1} and Lemma \ref{lem4.2} allows us to apply the dominated convergence theorem. So we have
\begin{equation*}
\lim_{t\rightarrow\infty}V_{\text{mix}}(K_t)=V_{\text{mix}}(K_{\infty})\geq\frac{1}{2}.
\end{equation*}
Thus by the assumption that the sequence of $K_j$ are the solutions of the equation \eqref{E:87}, we also have
\begin{equation}\label{E:81}
V_{\text{mix}}(K_{j})\geq\frac{1}{2}.
\end{equation}
On the other hand, by the assumption, we have for some $u_0\in\mathbb{S}^{n-1}$,
\begin{equation*}
\lim_{j\rightarrow\infty}h_{K_j}(u_0)=0.
\end{equation*}
Hence, for any $\varepsilon>0$ small and sufficiently large $j$, $h_{K_j}(u_0)<\varepsilon$. This means that
\begin{equation}\label{E:86}
K_j\subset\{y\in\mathbb{R}^n; \langle y, u_0\rangle\leq\varepsilon\}.
\end{equation}
We also know that $K_j\subset B_R$ for some $R>0$ from the known upper bound of $h_{K_j}$. Combining this with \eqref{E:86}, we get
\begin{equation}\label{E:85}
K_j\subset B_R\cap\{y\in\mathbb{R}^n: \langle y, u_0\rangle\leq\varepsilon\}.
\end{equation}
From Proposition \ref{prop1} and the assumption that $K_j\in\mathcal{K}^n_{e}$, we obtain for any $\varepsilon>0$ small enough,
\begin{equation*}
V_{\text{mix}}(K_j)=2V_{\text{mix}}(K_j^{+})< V_{\text{mix}}( B_R^{+}\cap\{y\in\mathbb{R}^n; \langle y, u_0\rangle_{+}\leq\varepsilon\})\leq V_{\text{mix}}(\{y\in\mathbb{R}^n; \langle y, u_0\rangle_{+}\leq\varepsilon\})\rightarrow0,
\end{equation*}
which is a contradiction to \eqref{E:81}.

\end{proof}

\section*{Funding Declarations} The first named author is supported by NSF DMS grant 2402038. 

\section*{Acknowledgments}
We would like to express our sincere gratitude to Professor Yong Huang for his guidance, support, and useful advice. We would also like to thank Jinrong Hu for the helpful discussions.

\end{document}